\pdfoutput=1
\documentclass[10pt]{amsart}

\numberwithin{equation}{section}

\usepackage{amsmath,amsfonts,amssymb}
\usepackage{mathtools}
\usepackage[utf8]{inputenc}
\usepackage[english]{babel}
\usepackage{hyperref}

\DeclareUnicodeCharacter{FB01}{fi}
\newtheorem{theorem}{Theorem}[section]

\theoremstyle{definition}
\newtheorem{definition}{Definition}

\begin{document}
\title[HITCHIN FIBRATION FOR SYMPLECTIC AND ORTHOGONAL PARABOLIC BUNDLES]
{HITCHIN FIBRATION ON MODULI OF SYMPLECTIC AND ORTHOGONAL PARABOLIC HIGGS BUNDLES}

\author{Sumit Roy}
\address{School of Mathematics, Tata Institute of Fundamental Research, Homi Bhabha Road, Colaba, Mumbai 400005, India.}
\email{sumit@math.tifr.res.in}
\thanks{E-mail : sumit@math.tifr.res.in}
\thanks{Address : School of Mathematics, Tata Institute of Fundamental Research, Homi Bhabha Road, Colaba, Mumbai 400005, India.}

\begin{abstract}
	Let $X$ be a compact Riemann surface of genus $g \geq 2$, and let $D \subset X$ be a fixed finite subset. Let $\mathcal{M}(r,d,\alpha)$ denote the moduli space of stable parabolic $G$-bundles (where $G$ is a complex orthogonal or symplectic group) of rank $r$, degree $d$ and weight type $\alpha$ over $X$. Hitchin, in his paper \cite{Hi87} discovered that the cotangent bundle of the moduli space of stable bundles on an algebraic curve is an algebraically completely integrable system fibered, over a space of invariant polynomials, either by a Jacobian or a Prym variety of spectral curves. In this paper we study the Hitchin fibers for $\mathcal{M}(r,d,\alpha)$.
\end{abstract}
\maketitle
\section{Introduction}
Let $X$ be a compact Riemann surface, and let $D \subset X$ be a  fixed finite subset. The notion of parabolic bundles over a curve and their moduli space was described in \cite{MS80}. In \cite{BR89}, Bhosle and Ramanathan extended this notion to parabolic $G$-bundles where $G$ is a connected reductive group. The moduli space of parabolic Higgs bundles was constructed by Yokogawa \cite{Y93}. 

The notion of symplectic and orthogonal parabolic bundles were described in \cite{BMW11}. Those bundles are parabolic vector bundles with a nondegenerate (in a suitable sense) symmetric or anti-symmetric form taking values in a parabolic line bundle. When all parabolic weights are rational, the notion of parabolic symplectic or othogonal bundles coincides with the notion of parabolic principal $G$-bundles where $G$ is an orthogonal or symplectic group. \
 
 Hitchin in his paper \cite{Hi87} showed that the moduli space of stable $G$-Higgs bundles (where $G = \mathrm{GL}(m,\mathbb{C}), \mathrm{Sp}(2m,\mathbb{C})$, $\mathrm{SO}(2m,\mathbb{C})$ or  $\mathrm{SO}(2m+1,\mathbb{C} $)) on an algebraic curve forms an algebraically completely integrable system fibered, over a space of invariant polynomials, either by a Jacobian or a Prym variety of spectral curves. Later, in \cite{M94}, Markman proved the same for the moduli space of stable $L$-twisted Higgs bundles where $L$ is a positive line bundle on $X$ satisfying $L \geq K_X$. \ 
 
 In this paper, we will prove that the Hitchin fibers for the moduli space of stable parabolic symplectic or othogonal Higgs bundles on an algebraic curve are Prym varieties of the spectral curve with respect to an involution. In our context, Higgs fields are strongly parabolic, meaning that the Higgs field is nilpotent with respect to the flag. \
 
 Here is a brief outline of this paper:\\
 In section $2$ we give the necessary details regarding parabolic symplectic or othogonal Higgs bundles and their moduli. In section $3$ we give a description of the Hitchin fibration and the spectral data. In section $4$ we prove the main result for three different cases, i.e. for $G$ = $\mathrm{Sp}(2m,\mathbb{C})$,  $\mathrm{SO}(2m,\mathbb{C})$ and $\mathrm{SO}(2m+1,\mathbb{C})$. For $G$ = $\mathrm{Sp}(2m,\mathbb{C})$ or $\mathrm{SO}(2m+1,\mathbb{C})$, fibers are Prym variety of the spectral curve with respect to an involution wih fixed points. For $G$ = $\mathrm{SO}(2m,\mathbb{C})$, the spectral curve is singular but the fibers are Prym variety of the desingulrised spectral curve with respect to an involution without fixed points. The main results of section $4$ are proven in Theorem \ref{thm1}, Theorem \ref{thm2} and Theorem \ref{thm3}.

\section{Preliminaries}
\subsection{Parabolic Vector Bundles}Let $X$ be a compact Riemann surface of genus $g \geq 2$ with $n$ distinct marked points $p_1,..., p_n$. Let $D= p_1 + \cdots + p_n$ be the effective reduced divisor. A \textit{parabolic vector bundle} $E_*$ on $X$ is a holomorphic vector bundle $E$ of rank $r$ over $X$ together with a parabolic structure, i.e. for every point $p \in D$, we have
\begin{enumerate}
     \item a filtration of subspaces  $$E_p=: E_{p,1}\supsetneq \dots \supsetneq E_{p,r(p)} \supsetneq 0, $$
     \item a sequence of real numbers (parabolic weights)  satisfying $$0\leq \alpha_1(p) < \alpha_2(p) < \dots < \alpha_r(p) < 1.$$
\end{enumerate}
 The parabolic structure is said to have \textit{full flags} whenever dim$(E_{p,i}/E_{p,i+1}) = 1$ \hspace{1cm}$\forall i, \forall p\in D$. \\
 The \textit{parabolic degree} of a parabolic vector bundle $E_*$ is defined as
\[
\operatorname{par-deg}(E_*):= \deg(E)+ \sum\limits_{p\in D}\sum\limits_{i} \alpha_i(p) \cdot \dim(E_{p,i}/E_{p,i+1})
\]
and the real number par-deg$(E_*)$/rank$(E_*)$ is called the \textit{parabolic slope} of $E_*$ and it is denoted by $\mu_{par}(E_*)$.\\
The \textit{dual} of a parabolic bundle and \textit{tensor product} of two parabolic bundles can be defined in a natural way (see \cite{Y95}).\\
A \textit{parabolic homomorphism} $\phi : E_* \to E^\prime_*$ between two parabolic bundles is a homomorphism of vector bundles that satisfies
the following : at each $p \in D$ we have $\phi_p(E_{p,i}) \subset E_{p,i+1}^\prime$ whenever $\alpha_i(p) > \alpha_{i+1}^\prime(p)$. Furthermore, we call such map \textit{strongly parabolic} if $\alpha_i(p) \geq \alpha_{i+1}^\prime(p)$ implies $\phi_p(E_{p,i}) \subset E_{p,i+1}^\prime$ for every $p \in D$.

\subsection{Orthogonal and Symplectic Parabolic Bundles} Fix a parabolic line bundle $L_*$. Let $E_*$ be a parabolic vector bundle and let
\[
\varphi : E_* \otimes E_* \to L_*
\]
be a homomorphism of parabolic bundles. Tensoring both sides with the parabolic dual $E^*_*$ we get a homomorphism
\[
\varphi \otimes Id : E_* \otimes E_* \otimes E^*_* \to L_* \otimes E^*_*.
\]
The trivial line bundle $\mathcal{O}_X$ equipped with the trivial parabolic structure (meaning parabolic weights are all zero) is realized as a parabolic subbundle of $E_* \otimes E^*_*$. Let
\[
\tilde{\varphi} : E_* \to L_* \otimes E^*_*
\]
be the homomorphism defined by the composition
\[
E_* = E_* \otimes \mathcal{O}_X  \xhookrightarrow{} E_* \otimes (E_* \otimes E^*_*) = (E_* \otimes E_*) \otimes E^*_* \xrightarrow{\varphi \otimes Id} L_* \otimes E^*_*.
\]
\begin{definition} A \textit{parabolic symplectic bundle} is a pair $(E_*,\varphi)$ of the above form such that $\varphi$ is anti-symmetric and the homomorphism $\tilde{\varphi}$ is an isomorphism. \

A \textit{parabolic orthogonal bundle} is a pair $(E_*,\varphi)$ of the above form such that $\varphi$ is symmetric and the homomorphism $\tilde{\varphi}$ is an isomorphism.

\end{definition}
 
\subsection{Parabolic Higgs Bundles} Let $K$ be the canonical bundle on $X$. We write $K(D) \coloneqq K \otimes \mathcal{O}(D)$. A \textit{parabolic Higgs bundle} on $X$ is a parabolic bundle $E_*$ on $X$ together with a Higgs field $\Phi : E_* \to E_* \otimes K(D)$ such that $\Phi$ is strongly parabolic.\

Higgs field associated to a parabolic Higgs bundle is called \textit{parabolic Higgs field}.
 
\subsection{Orthogonal or Symplectic Parabolic Higgs Bundles} Let $(E_*,\varphi)$ be an orthogonal or symplectic parabolic bundle on $X$. A parabolic Higgs field on $E_*$ will induce a parabolic Higgs field on $L_* \otimes E^*_*$. A parabolic Higgs field $\Phi$ is said to be compatible with  $\varphi$ if $\tilde{\varphi}$ takes $\Phi$ to the induced parabolic Higgs field on $L_* \otimes E^*_*$.\
 
 An \textit{orthogonal} (resp. \textit{symplectic}) \textit{parabolic Higgs bundle} $(E_*,\varphi,\Phi)$ is an orthogonal (resp. symplectic) parabolic bundle $(E_*,\varphi)$  togther with a parabolic Higgs field $\Phi$ on $E_*$ which is compatible with $\varphi$.
 
\begin{definition} A holomorphic subbundle $F \subset E$ is called \textit{isotropic} if $\varphi(F \otimes F) = 0$.\\
 An orthogonal or symplectic parabolic Higgs bundle $(E_*,\varphi,\Phi)$ will be called \textit{stable} (resp. \textit{semistable}) if for every isotropic subbundle $F \subset E$ of positive rank, the following condition holds
\[
\mu_{par}(F_*) < \mu_{par}(E_*) \hspace{0.4cm}(\text{resp.} \hspace{0.15cm} \mu_{par}(F_*) \leq \mu_{par}(E_*)).
\]
\end{definition}\

When all parabolic weights are rational, the notion of orthogonal or symplectic parabolic bundles coincides with the notion of parabolic principal $G$-bundles where $G$ is an orthogonal or symplectic group, respectively (see [1], [2], [3] for more details).\
\subsection{Moduli Space of orthogonal or symplectic parabolic Higgs bundles} Fix a parabolic line bundle $M_*$ with trivial parabolic structure and let $E_*$ be a parabolic orthogonal or symplectic Higgs bundle of rank $r$, degree $d$ w.r.t. the line bundle $M_*$. So, we have an isomorphism $E_* \cong E_*^*  \otimes M_*$. Whenever $r$ is odd, $\deg M_* $ is even. When we deal with parabolic $\mathrm{SO}(2m+1)$-Higgs bundles, we assume that the parabolic determinant (top parabolic exterior power) is $M_*^{2m+1/2}$ after choosing a square root of $M_*$.

Moduli space of semistable parabolic $G$-bundles of rank $r$ and degree $d$ and fixed parabolic structure $\alpha$ was described in \cite{BBN01} and \cite{BR89}. It is a normal projective variety. When the parabolic structure $\alpha$ have full flags, the moduli space $\mathcal{M}(r,d,\alpha)$ of stable parabolic $G$-bundles is of dimension
\[
\dim Z(G) +  (g-1)\dim(G) + n\dim(G/B),
\]
where $Z(G)$ denotes the the center of $G$ and $n$ is the number of marked points. The last summand comes from the fact that the flags we are considering over each point of $D$ are full flags and $B$ is the Borel subgroup of $G$ determined by $\alpha$. From now on we assume that the parabolic weights are rational and the parabolic structures over all marked points have full flags. In case of orthogonal or symplectic groups,
\[
\dim\mathcal{M}(r,d,\alpha) = (g-1)\dim(G) + n\dim(G/B),
\]
where $G$ is an orthogonal or symplectic group respectively.\\
The moduli space $\mathcal{N}(r,d,\alpha)$ (see \cite{R16} for the construction) of stable orthogonal or symplectic parabolic Higgs bundles of rank $r$, degree $d$ and weight type $\alpha$ is a smooth irreducible complex variety. Serre duality of parabolic vector bundles (see \cite{Y95},\cite{BY96}) gives us an embedding of the cotangent bundle $T^*\mathcal{M}(r,d,\alpha) \xhookrightarrow{} \mathcal{N}(r,d,\alpha)$ as an open subset. The natural symplectic structure on the cotangent bundle extends to $\mathcal{N}(r,d,\alpha)$ (see Biswas–Ramanan \cite{BR94}, Konno \cite{K93} for details). Hence,
\[
\dim\mathcal{N}(r,d,\alpha) = 2\dim \mathcal{M}(r,d,\alpha) = 2(g-1)\dim(G) + 2n\dim(G/B).
\]

From now on, we shall  denote a parabolic symplectic or orthogonal Higgs bundle simply by $E$ instead of $E_*$.
\section{The hitchin system and spectral data}
\subsection{Hitchin Map}

We will now describe the Hitchin map and the Hitchin space for orthogonal or symplectic parabolic Higgs bundles.\

Let $\mathcal{S}$ be the total space of $K(D)$ and $p : \mathcal{S} \to X $ be the natural projection and $x \in H^0(\mathcal{S}; p^*K(D))$ be the tautological section. The characteristic polynomial of a Higgs field $\Phi$ gives
\[
 \det(x.\text{Id}  - p^*\Phi) = x^r + \tilde{s_1}x^{r-1} + \cdots +\tilde{s_r},
\]
where $\tilde{s_i} = p^*s_i$ for some $s_i \in H^0(X; K^i(D^i))$ and $K^i(D^j)$ denotes the tensor product of $i$-th power of $K$ and $j$-th power of $\mathcal{O}(D)$.\

Since $\Phi$ is strongly parabolic, its residue at each parabolic point is nilpotent and hence $s_i \in H^0(X; K^i(D^{i-1}))$. Therefore we have the \textit{Hitchin map}
\[
h : \mathcal{N}(r,d,\alpha) \longrightarrow \label{key}\mathcal{H} \coloneqq \oplus_{i=1}^r H^0(X; K^i(D^{i-1}))
\] 
which sends a parabolic symplectic or orthogonal Higgs bundle to the coefficients of the characteristic polynomial of its Higgs field and $\mathcal{H}$ is called the \textit{Hitchin base}. In fact, this is the definition for parabolic Higgs bundles. We will see later that in our case, the odd coefficients of the characteristic polynomial are all zero, i.e. for odd $i$, $s_i = 0$.\

This morphism doesn't depend on the parabolic stucture at each parabolic points, as it only depends on $\Phi$ and $K(D)$. This is a proper map (see \cite{M94} for details). Furthermore, we will see that $\mathcal{H}$ satisfies $\dim\mathcal{H} = \dim\mathcal{N}(r,d,\alpha)/2$, making the parabolic orthogonal or symplectic Higgs bundle moduli space into an integrable system. 
\subsection{Spectral Curves} Given $s=(s_1,....,s_r)\in \mathcal{H}$ with $s_i \in H^0(X; K^i(D^{i-1})) \subset  H^0(X; K(D)^i)$, we think of $s_i$ as a section of $K(D)^i$ and then we can define a \textit{spectral curve} $X_s$ in $\mathcal{S}$ as follows : 
consider the tautological section $x \in H^0(\mathcal{S}; p^*K(D))$ of $p^*K(D)$. Then $X_s$ is the zero locus of
\[
x^r + s_1x^{r-1} + \cdots + s_r.
\]

When $X_s$ is reduced, the projection $\pi = \left.p\right|_{X_s} : X_s \to X$ is a ramified covering of degree $r$. For a generic $s\in \mathcal{H}$, $X_s$ is smooth (see \cite{BNR89}). The genus of $X_s$ can be calculated using the adjunction formula :

\begin{align*}
\begin{split}
\label{eqn}
2g(X_s) - 2 &= \deg(K_{X_s})\\
            &= K_{X_s}.X_s\\
            &=K_S.X_s + X_s^2\\
            &=rc_1(\mathcal{O}(-D)) + K(D)^2\\
            &=-rn + r^2(2g-2+n).
\end{split}
\end{align*}
So,
\begin{equation}\label{genus}
g(X_s) = \dfrac{-rn + r^2(2g-2+n) + 2}{2}.
\end{equation}
\subsection{Pullback Bundle} Let us assume that $X_s$ is smooth and pull back the Higgs bundle $(E,\Phi)$ to $X_s$ by $\pi$. We canonically get a line bundle $L$ which satisfies the following exact sequence
\[
0 \to L(-R) \to \pi^*E \xrightarrow{\pi^*\Phi - \otimes x} \pi^*(E \otimes K(D)) \to L \otimes \pi^*K(D) \to 0
\]
(see \cite[Proposition 3.6]{BNR89}) where $R$ is the ramification divisor of the covering $\pi : X_s \to X$. Moreover we  can recover the Higgs field on $\pi_*L$ as follows : 
multiplication by the tautological section $x$ of $\pi^*K(D)$ gives us a morphism
\[
\Phi : \pi_*L \to \pi_*(L \otimes \pi^*K(D)) = \pi_*L \otimes K(D).
\] 
In order to obtain the degree of $L$, apply Grothendieck-Riemann-Roch theorem to the morphism $\pi : X_s \to X$, we get 
\[
\deg(\pi_*L) + r(1 - g) = \deg(L) + (1 - g(X_s))
\]
and hence
\begin{align}
\begin{split}
\label{degree}
\deg(L) &= \deg(\pi_*L) + r(1 - g) + (g(X_s) - 1).
\end{split}
\end{align}
Also, following \cite{Hi87} we have an exact sequence of sheaves 
\begin{equation}\label{sheaves}
0 \to \mathcal{O}(E^*) \to \mathcal{O}(\pi_*L)^* \to \mathcal{S} \to 0
\end{equation}
where $\mathcal{S}$ is a sheaf supported at the branch points of $\pi : X_s \to X$.\\The ramification divisor on $X_s$ is the divisor of a section of $K_{X_s} \otimes \pi^*K^*$ which has degree $(2g(X_s) - 2) - r(2g - 2)$. Thus from $(3.3)$, since $\mathcal{S}$ is supported on the branch points,
\begin{align*}
\deg(E^*) &= \deg(\pi_*L)^* - \deg(\mathcal{S})\\
          &= r(g - 1) - (g(X_s) - 1) - \deg(L). \hspace{1cm}  (\text{by}\hspace{0.1cm} \ref{degree})
\end{align*}
Hence
\begin{equation}
\deg(L) = r(g - 1) + (1 - g(X_s)) - \deg(E^*).
\end{equation}
\section{Proof of the theorems}
Now we will describe the fibers of the Hitchin map for orthogonal or symplectic parabolic Higgs bundles. We will do this for three different cases, i.e. when $G$ = $\mathrm{Sp}(2m,\mathbb{C})$,  $\mathrm{SO}(2m,\mathbb{C})$ and $\mathrm{SO}(2m+1,\mathbb{C})$.
\subsection{CASE I ($G$ = $\mathrm{Sp}(2m,\mathbb{C})$)}\label{caseI} A point of $\mathcal{N}(2m,d,\alpha)$ now consists of a stable symplectic parabolic Higgs bundle of rank $2m$ with a fixed parabolic line bundle $M$. We assume that $M$ has trivial parabolic structure. A point of $\mathcal{N}(2m,d,\alpha)$ (or simply $\mathcal{N}_{\mathrm{Sp}}(2m)$) can be viewed as a stable parabolic bundle $E$ of rank 2m with a nondegenerate symplectic form $< , >$, together with a holomorphic section $\Phi \in H^0(X; \operatorname{End}(E) \otimes K(D))$ which satisfies
\[
<\Phi v,w> = - <v,\Phi w>.
\]

Suppose $A\in \mathfrak{sp} (2 m,\mathbb{C})$ has distinct eigenvalues $\lambda_i$'s and $v_i$, $v_j$ are eigenvectors of $A$ with eigenvalues $\lambda_i$ and $\lambda_j$. Then
\[
\lambda_i<v_i,v_j> = <Av_i,v_j> = -<v_i,Av_j> = -\lambda_j<v_i,v_j>.
\]
So, $<v_i,v_j> = 0$ unless $\lambda_i = - \lambda_j$. Hence it follows from the nondegeneracy of symplectic form that if $\lambda_i$ is an eigenvalue then $-\lambda_i$ is also an eigenvalue. Thus the characteristic polynomial looks like
\[
\det(\lambda - A) = \lambda^{2m} + s_2\lambda^{2m-2} + \cdots + s_{2m},
\]
where the polynomials $s_2,...,s_{2m}$ form a basis for the invariant polynomials on $\mathfrak{sp}(2m,\mathbb{C})$. So, the Hitchin map is 
\[
h : \mathcal{N}_{\mathrm{Sp}}(2m) \longrightarrow \mathcal{H} = \oplus_{i=1}^m H^0(X; K^{2i}(D^{2i-1})).
\]
Since $\mathcal{H}$ is generated by $s_2,...,s_{2m}$, $\dim\mathcal{H} = \sum_{i=1}^{m} h^0(X; K^{2i}(D^{2i-1}))$. Now by Riemann-Roch and the parabolic Serre duality, $h^0(X; K^{2i}(D^{2i-1})) = 2i(2g-2) + (2i-1)n + 1 - g$. So, $\dim\mathcal{H} = m(m+1)(2g-2+n) + m(1-g-n) = m(2m+1)(g-1) + m^2n$.\\
In this case, the genus of the spectral curve $X_s$ (follows from (\ref{genus})) for a generic point $s \in \mathcal{H}$ is
\[
g(X_s) = -mn + 2m^2(2g-2+n) + 1 .
\]
Also, the spectral curve $X_s$ is given by the equation
\[
x^{2m} + s_2 x^{2m-2} + \cdots + s_{2m} = 0
\]
and possesses the involution $\sigma(\eta) = -\eta $ (since all odd coefficients of the equation are zero). Thus one can define a $2$-fold cover
\[
q : X_s \to X_s/\sigma.
\]
The involution acts on the line bundles of degree zero over $X_s$. The \textit{Prym variety} Prym$(X_s, X_s/\sigma)$ is given by the line bundles $L \in \text{Jac}(X_s)$ such that $\sigma^*L \cong L^*$. The dimension of the Prym variety is $g(X_s) - g(X_s/\sigma)$. \\
Note that the fixed points of the involution are the intersection points of zeroes of $s_{2m} \in H^0(X; K(D)^{2m})$ and the zero section $x = 0$. So $\sigma$ has $2m(2g - 2 + n)$ fixed points since they are the zeroes of a section of $K(D)^{2m}$ on $X$.
By Riemann-Hurwitz formula, we have 
\[
2g(X_s) - 2 = 2(2g(X_s/\sigma) - 2) + 2m(2g - 2 + n).
\]
So,
\begin{align}
\begin{split}
\label{eqn}
\dim \text{Prym}(X_s, X_s/\sigma) &= g(X_s) - g(X_s/\sigma)\\
                      &= g(X_s) - \dfrac{g(X_s)}{2} - \dfrac{1}{2} + \dfrac{m}{2}(2g - 2 + n)\\
                      &= \dfrac{1}{2}g(X_s) - \dfrac{1}{2} + \dfrac{m}{2}(2g - 2 + n)\\
                      &= -\dfrac{mn}{2} + m^2(2g - 2 + n) + \dfrac{m}{2}(2g - 2 + n)\\
                      &= m(2m + 1)(g - 1) + m^2n\\
                      &= \dim \text{Sp}(2m,\mathbb{C})(g - 1) + n\dim (\text{Sp}(2m,\mathbb{C})/B) \\
                      &= \dim \mathcal{M}(2m,d,\alpha)\hspace{0.2cm} (\text{or}\hspace{0.1cm} \dim \mathcal{M}_{\mathrm{Sp}}(2m))                     
\end{split}
\end{align}
since the dimension of a Borel subgroup $B$ of $\text{Sp}(2m,\mathbb{C})$ is $m^2 + m$. So, $\dim\mathcal{H} = \dim\mathcal{M}_{\mathrm{Sp}}(2m) = \dim\mathcal{N}_{\mathrm{Sp}}(2m)/2 = \dim \text{Prym}(X_s, X_s/\sigma).$
\begin{theorem}\label{thm1}
\textit{If $X_s$ is smooth, the generic fibers $h^{-1}(s)$ of the Hitchin map for parabolic $\mathrm{Sp}(2m,\mathbb{C})$-Higgs bundles is given by Prym varieties Prym$(X_s, X_s/\sigma)$}.
\end{theorem}
\begin{proof} Let $E$ be a parabolic symplectic Higgs bundle. So,
\begin{align*}
\operatorname{par-deg}(E) &= \operatorname{par-deg}(E^* \otimes M) \\
                          &= \operatorname{par-deg}(E^*) + 2m\deg(M)\\
                          &= - \operatorname{par-deg}(E) + 2m\deg(M).
\end{align*}
i.e. 
\[
\operatorname{par-deg}(E) = m\deg(M).
\]

 Since $X_s$ is smooth, we obtain an eigenspace bundle $L \subset \ker(\eta - \Phi) \subset \pi^*E$  corresponding to the eigenvalue $\eta$. Then $\sigma^*L$ is the eigenspace bundle corresponding to $- \eta$.
 So the symplectic form on $E$ defines a section of $L^* \otimes \sigma^*L^* \otimes \pi^*M$ which is non-vanishing if the eigenvalues are distinct, i.e. away from the ramification locus of $X_s$. Thus from (\ref{sheaves}), we have
\begin{align*}
 \operatorname{par-deg}(E^*) &= \deg(\pi_*L)^* - \deg(\mathcal{S})\\
                             &= -\dfrac{1}{2}\deg(K_{X_s} \otimes \pi^*K^*) -\deg(L).
\end{align*}
So,
\[
\deg(L) = -\dfrac{1}{2}\deg(K_{X_s} \otimes \pi^*K^*) + \dfrac{1}{2}\deg(\pi^*M).
\]
Thus
\[
\deg(L^* \otimes \sigma^*L^* \otimes \pi^*M) = \deg(K_{X_s} \otimes \pi^*K^*).
\]
Also, $L^* \otimes \sigma^*L^* \otimes \pi^*M$ has a section with zeros on the ramification locus of $X_s$, a divisor of $K_{X_s} \otimes \pi^*K^*$. So, 
\[
L^* \otimes \sigma^*L^* \otimes \pi^*M  \cong K_{X_s} \otimes \pi^*K^*.
\]
So,
\begin{equation}\label{keyeqn}
\sigma^*L \cong L^* \otimes (K_{X_s} \otimes \pi^*K^*)^{-1} \otimes \pi^*M.
\end{equation}
 Then choosing a holomorphic square root $R = (K_{X_s}\otimes \pi^*K^* \otimes \pi^*M^*)^{1/2}$ and setting $U = L \otimes R$, we obtain a point $U \in \text{Prym}(X_s, X_s/\sigma)$, i.e. $\sigma^*U \cong U^*$.\
 
Conversely, suppose $U \in \text{Prym}(X_s, X_s/\sigma)$. Then consider the line bundle 
\[
L = U \otimes (K_{X_s} \otimes \pi^*K^* \otimes \pi^*M^*)^{-1/2}
\]
and the parabolic structure on $E = \pi_*L$ is defined as follows : the Higgs field $\Phi$ is given by multiplication by the tautological section $x$.\

For each $p \in D$, there is an open subset $A$ of $X$ as in \cite[Proposition 2.2]{LM10} such that (as an $\mathcal{O}_A $-module)
\[
\left.E\right|_A = \mathcal{O}_A[x]/(x^{2m} + s_2x^{2m-2} + \cdots + s_{2m}).
\]
Since $s_i$ vanishes at $p$ for all $i$, we have 
\[\left.E\right|_p = \mathbb{C}[x]/(x^{2m}).
\]
That defines a full flag on $\left.E\right|_p$ as $\Phi$ is given by multiplication by $x$ which must coincide with the parabolic structure.\
Since the smoothness of the spectral curve $X_s$ guarantees that there are no $\Phi$- preserved subbundle of $L$, so $E$ is stable.\

Now the symplectic structure on $E$ is defined as follows. The involution $\sigma$ produces a bilinear form on $E = \pi_*L$ away from the ramification locus of $X_s$. For two sections $v,w \in H^0(\pi^{-1}(V); L)$, the form 
\begin{equation}\label{eqn}
<v,w> = \operatorname{tr}_{X_s/X}(\frac{\sigma^*(v)w}{d\pi}) \in H^0(\pi^{-1}(V); \pi^*M)
\end{equation}
is non-degenerate by (\ref{keyeqn}),  where $d\pi$ is the canonical section given by the derivative of $\pi : X_s \to X$.
This pairing is skew because of the fact that $(\sigma^*)^2 = - 1$ on $L$, see \cite[5.10]{Hi87} for details.\\
Hence the generic fibers of the corresponding Hitchin fibration is identified with $\text{Prym}(X_s, X_s/\sigma)$.
\end{proof}
\subsection{CASE II ($G= \mathrm{SO}(2m,\mathbb{C})$)} We shall now consider the moduli space of stable parabolic orthogonal Higgs bundles with a fixed parabolic line bundle $M$, where $M$ has trivial parabolic structure. A point of $\mathcal{N}_{\mathrm{SO}}(2m)$ is now  a rank $2m$ stable vector bundle $E$ with a non-degenerate symmetric bilinear  form $< , >$, and $\Phi \in H^0(X; \operatorname{End}(E) \otimes K(D))$ which satisfies $<\Phi v,w> =  -<v,\Phi w>$. \\
As in the previous case, one can see that for a matrix $A \in \mathfrak{so}(2m, \mathbb{C})$ with distinct eigenvalues, if $\lambda_{i}$ is an eigenvalue then so is $- \lambda_{i}$. So the characteristic polynomial is of the form 
\[
\det(\lambda - A) = \lambda^{2m} + s_2\lambda^{2m-2} + \cdots + s_{2m}.
\]
In this case, the coefficient $s_{2m}$ is a square of a polynomial $p_m$, the Pfaffian, of degree $m$. A basis for the invariant polynomials on the Lie algebra $\mathfrak{so}(2m,\mathbb{C})$ is given by the coefficients $s_2,...,s_{2m-2},p_m$. In this case, the Hitchin map is given by
\[
h : \mathcal{N}_{\mathrm{SO}}(2m) \longrightarrow \mathcal{H} = \oplus_{i=1}^m H^0(X; K^{2i}(D^{2i-1}))
\]
and $\mathcal{H}$ is generated by $s_2,...,s_{2m-2},p_m$ where $s_{2i} \in H^0(X; K(D)^{2i})$ and $p_m \in H^0(X; K(D)^m)$. In this case, $\dim\mathcal{H} = \sum_{i=1}^{m-1} h^0(X; K^{2i}(D^{2i-1})) + h^0(X; K(D)^m)$. Again, using the Riemann-Roch theorem, $\dim\mathcal{H} = m(2m-1)(g-1) + mn(m-1)$.
Here the spectral curve $X_s$ given by the equation
\[
 \det(x - \Phi) = x^{2m} + s_2x^{2m-2} + \cdots +  s_{2m-2}x^2 + p_m^2.
\] 
This curve has singularities where $x = 0$ and $p_m = 0$. Hence, by Bertini's theorem the generic divisor has these as only singularities as these curves are the base points of this system. Since $p_m$ is a section of $K(D)^m$, there are $\deg K(D)^m = m(2g - 2 + n)$ singularities. The \textit{virtual genus} of $X_s$ for a generic $s \in \mathcal{H}$ is given by $g(X_s) = -mn + 2m^2(2g - 2 + n) + 1$. So the genus of the non-singular model $\hat{X_s}$ is thus 
\begin{align*}
g(\hat{X_s}) &= g(X_s) - \text{number of singularities}\\
                 &= -mn + 2m^2(2g - 2 + n) + 1 - m(2g - 2 + n) \\
                 &= 2m(2m - 1)(g - 1) + 2mn(m - 1) + 1
\end{align*}
Fixed points of the involution $\sigma(\eta) = - \eta$ on $X_s$ are the singularities of $X_s$ and so extends to an involution $\hat{\sigma}$ on $\hat{X_s}$ without fixed points. So by Riemann-Hurewitz, 
\[
2g(\hat{X_s}) - 2 = 2(2g(\hat{X_s}/\hat{\sigma}) - 2)
\]
and thus 
\begin{align*}
\dim\text{Prym}(\hat{X_s},\hat{X_s}/\hat{\sigma}) &= g(\hat{X_s}) - g(\hat{X_s}/\hat{\sigma})\\
                                  &= \dfrac{1}{2}g(\hat{X_s}) - \dfrac{1}{2}\\
                                  &= m(2m - 1)(g - 1) + mn(m - 1) \\
                                  &= \dim \mathrm{SO}(2m, \mathbb{C}) + n\dim(\mathrm{SO}(2m,\mathbb{C})/B) \\
                                  &= \dim\mathcal{M}_{\mathrm{SO}}(2m)
\end{align*}
since the dimension of a Borel subgroup $B$ of $\mathrm{SO}(2m,\mathbb{C})$ is $m^2$ and hence $\dim\mathcal{H} = \dim\mathcal{M}_{\mathrm{SO}}(2m) = \dim\mathcal{N}_{\mathrm{SO}}(2m)/2 = \dim\text{Prym}(\hat{X_s},\hat{X_s}/\hat{\sigma}).$
\begin{theorem}\label{thm2}
\textit{The smooth fibers $h^{-1}(s)$ of the Hitchin fibration for parabolic $\mathrm{SO}(2m,\mathbb{C})$-Higgs bundles are given by Prym$(\hat{X_s},\hat{\sigma})$, where $\hat{X_s}$ is the desingularisation of the curve $X_s$. }
\end{theorem}
\begin{proof} Let $E$ be a parabolic $\mathrm{SO}(2m,\mathbb{C})$-Higgs bundle. Since $\hat{X_s}$ is smooth, one can obtain an eigenspace bundle $L \subset \ker(\eta - \Phi) \subset \pi^*E$ corresponding to an eigenvalue $\eta$. From  (\ref{keyeqn}), we have
\[
\hat{\sigma}^*L \cong L^* \otimes (K_{\hat{X_s}}\otimes \pi^*K^*)^{-1} \otimes \pi^*M
\]
and setting $U = L \otimes (K_{\hat{X_s}}\otimes \pi^*K^* \otimes \pi^*M^*)^{1/2}$,  we obtain a point of Prym$(\hat{X_s},\hat{X_s}/\hat{\sigma})$.\\
Conversely, as before, a line bundle $U \in \text{Prym}(\hat{X_s},\hat{X_s}/\hat{\sigma})$ induces a parabolic Higgs bundle $E$, which is the direct image of $L = U \otimes (K_{\hat{X_s}}\otimes \pi^*K^* \otimes \pi^*M^*)^{-1/2}$. The non-degenerate pairing is given by
\[
<v,w> = \operatorname{tr}_{\hat{X_s}/X}(\frac{\sigma^*(v)w}{d\pi})
\]
as in (\ref{eqn}). This pairing is symmetric since $(\hat{\sigma}^*)^2$ is identity.
\end{proof}
\subsection{CASE III ($G = \mathrm{SO}(2m + 1,\mathbb{C})$)} Lastly, we consider a stable parabolic bundle $E$ of rank $(2m + 1)$ with a non-degenerate symmetric bilinear form $< , >$ with values in a fixed parabolic line bundle $M$ with trivial parabolic structure, and a Higgs field $\Phi \in H^0(X; \operatorname{End}(E) \otimes K(D))$ which satisfies 
\[
<\Phi v,w> = - < v,\Phi w>.
\] 
The characteristic polynomial of $\Phi$ is of the form (assuming $\Phi$ has distinct eigenvalues)
\[
\det(\lambda - \Phi) = \lambda(\lambda^{2m} + s_2\lambda^{2m - 2} + \cdots + s_{2m}).
\]
So, the vector bundle homomorphism $\Phi : E \to E \otimes K(D)$ always has an eigenvalue zero, so we will have a line bundle $E_0 \subset \ker\Phi$.\\
Therefore the bundle $E$ on $X$ is an extension
\begin{equation}\label{extn}
0 \to E_0 \to E \to E_1 \to 0.
\end{equation}
Using the symmetric bilinear form on $E$, the skew  map $\Phi$ will induce a section $\phi \in H^0(X; \wedge^{2}E^{*} \otimes M \otimes K(D))$. So,
\[
\phi^m \in H^0(X; \wedge^{2m}E^* \otimes M^m \otimes K(D)^m) \cong H^0(X; E \otimes M^{-1/2} \otimes K(D)^m).
\]
Where the isomorphism comes from the fact that $E \cong E^* \otimes M$ and $\det E \cong M^{(2m+1)/2}$. So, we get an injective morphism from $M^{1/2} \otimes K(D)^{-m}$ to $E$ whose image is the line bundle $E_0$. Thus, $E_0 \cong M^{1/2} \otimes K(D)^{-m}$ and hence from (\ref{extn}), we have 
\begin{equation}
\wedge^{2m}E_1 \cong K(D)^m \otimes M^m.
\end{equation}

The inner product $< , >$ on $E$ induces a non-degenerate skew form on $E_1$ with values in $K(D) \otimes M$, by
\[
 (v,w) := <\Phi v, w>.
\]
Since $(\Phi v,w) + (v,\Phi w) = <\Phi^2 v,w> + <\Phi v, \Phi w> = 0$, the form is well defined on $E_1$ but is singular when $s_{2m}=0$. If we consider $V = E_1 \otimes K(D)^{-1/2}$, then $V$ is a parabolic $\mathrm{Sp}(2m, \mathbb{C})$-Higgs bundle as in the CASE I (\ref{caseI}). Let us denote the induced homomorphism on $E_1$ by $\Phi$ for simplicity. The characteristic polynomial $\det(x - \Phi)$ of $\Phi$ defines a component of the spectral curve, which we shall denote by $\pi : X_s \to X$, and it is given by the equation
\[
x^{2m} + s_2x^{2m - 2} + \cdots + s_{2m} = 0,
\]
where $s_{2i} \in H^0(X; K(D)^{2i})$. This is a $2m$-fold cover of $X$ with genus $g(X_s)
 = -mn + 2m^2(2g - 2 + n) + 1$. As in the case of $\mathrm{Sp}(2m,\mathbb{C})$, the curve $X_s$ has an involution $\sigma(\eta)= - \eta$. Also,
 \begin{align*}
 \dim\text{Prym}(X_s,\sigma) &= m(2m + 1)(g - 1) + m^2n\\
                             &= \dim\mathrm{SO}(2m + 1,\mathbb{C})(g - 1) + n\dim(\mathrm{SO}(2m + 1,\mathbb{C})/B)\\
                             &= \dim\mathcal{M}_{\mathrm{SO}}(2m + 1)
 \end{align*}
 since the dimension of a Borel subgroup $B$ of $\mathrm{SO}(2m + 1,\mathbb{C})$ is $m^2 + m$.
 \begin{theorem}\label{thm3}
 \textit{The smooth fibers $h^{-1}(s)$ of the Hitchin fibration for a parabolic $\mathrm{SO}(2m + 1,\mathbb{C})$-Higgs bundles are given by Prym$(X_s,\sigma)$.}
 \end{theorem}
 \begin{proof} Let $E$ be a stable parabolic $\mathrm{SO}(2m + 1,\mathbb{C})$-Higgs bundle over $X$. From the previous discussion we will get a stable parabolic $\mathrm{Sp}(2m, \mathbb{C})$-Higgs bundle $V$ over $X$. Following the symplectic case (\ref{caseI}), we will get an element in the Prym variety of the spectral curve.\\ 	
Conversely, as in (\ref{caseI}), we can recapture the parabolic $\mathrm{Sp}(2m,\mathbb{C})$-Higgs bundle $V$ from a line bundle $L \in \text{Prym}(X_s,\sigma)$. Which will induce a rank $2m$ bundle $E_1$ equipped with a skew form $( , )$ with values in $K(D) \otimes M$ and a homomorphism $\Phi : E_1 \to E_1 \otimes K(D)$ satisfying $(\Phi v, w) = - (v, \Phi w)$. We need to reconstruct the parabolic $\mathrm{SO}(2m + 1, \mathbb{C})$-Higgs bundle $E$ from $E_1$.\\
Now we will follow the same technique as in \cite[5.17]{Hi87} in our case. For that, let's look at how $E_1$ arose and dualize the sequence ($\ref{extn}$) to get
\[
0 \to E_1^* \to E^* \to E_0^* \to 0.
\]
or,
\begin{equation}
0 \to E_1^* \otimes M \to E \to E_0^* \otimes M \to 0.
\end{equation}
Also, the composition $E_0 \to E \to E_0^* \otimes M$ defines a section of $E_0^* \otimes E_0^* \otimes M \cong K(D)^{2m}$ which is the coefficient $s_{2m}$. Since $E_0$ is zero on the zero set $Z$ of $s_{2m}$, we have an inclusion $E_0 \subset E_1^*$.\

Consider the exact sequence of sheaves
\[
0 \to \operatorname{Hom}(E_0^*\otimes M, E_1^* \otimes M) \xrightarrow{s_{2m}} \operatorname{Hom}(E_0, E_1^* \otimes M) \to \operatorname{Hom}_{Z}(E_0, E_1^* \otimes M) \to 0
\]
and then look at the cohomology sequence. The extension is defined by the coboundary map
\[
\delta : H^0(Z; \operatorname{Hom}(E_0, E_1^* \otimes M)) \to H^1(X; \operatorname{Hom}(E_0^* \otimes M, E_1^* \otimes M)).
\]
Now following \cite{Hi87}, we can construct an extension $E$
\begin{equation}
0 \to E_1^* \otimes M \to E \to \wedge^{2m}E_1 \otimes M \to 0
\end{equation}
by the coboundary map.\

We will use the above applied to the bundle $E_1$ arising from a point of the Prym variety of $X_s$. Since $E_1$ has a skew form $\omega \in H^0(X; E_1^* \otimes E_1^* \otimes K(D) \otimes M)$ and the homomorphism $\Phi \in H^0(X; \operatorname{End}E_1 \otimes K(D))$, considering 
\begin{align*}
\omega^{2m-1} &\in H^0(X; \wedge^{2m-1}E_1^* \otimes \wedge^{2m-1}E_1^* \otimes K(D)^{2m-1} \otimes M^{2m-1})\\
              &\cong H^0(X; E_1 \otimes E_1 \otimes (\wedge^{2m}E_1^*)^2 \otimes K(D)^{2m-1} \otimes M^{2m-1})\\
              &\cong H^0(X; E_1 \otimes E_1 \otimes K(D)^{-1} \otimes M^{-1})
\end{align*}
and then applying $\Phi$ to $\omega^{2m-1}$ to obtain a symmetric bilinear form 
\[
\Phi\omega^{2m-1} \in H^0(X; E_1 \otimes E_1 \otimes M^{-1}).
\]
This will induce a symmetric nondegenerate bilinear form on $E$ with values in $M$.

 \end{proof}

 \section*{Acknowledgement}
I would like to thank my advisor Prof. Indranil Biswas for suggesting the problem, and for helpful discussions. I would also like to thank Mr. Sujoy Chakraborty for useful discussions.

\end{document}